\documentclass[12pt]{article}
\usepackage{amsmath}
\usepackage{amssymb}
\usepackage[english]{babel}
\usepackage{amsthm}
\usepackage{graphicx}
\graphicspath{ {images/} }
\usepackage{caption}
\usepackage{subcaption}

\usepackage[bottom]{footmisc}

\newtheorem{theorem}{Theorem}
\newtheorem{proposition}[theorem]{Proposition}

\newtheorem{corollary}[theorem]{Corollary}

\theoremstyle{definition}

\newtheorem{remark}[theorem]{Remark}
\numberwithin{figure}{section}

\providecommand{\keywords}[1]
{
  \small	
  \textbf{\textit{Keywords: }} #1
}

\begin{document}
\baselineskip 12pt

\newcounter{mpFootnoteValueSaver}
\begin{center}
\textbf{\Large Anti-It\^o noise-induced phase transitions in tumor growth with chemotherapy} \\
\vspace{5mm}
\begin{tabular}{cc} \qquad \qquad Helder Rojas\footnotemark   & Luis Huamanchumo\footnotemark \\
[1.2\baselineskip]
\multicolumn{2}{c}{Escuela Profesional de Ingeniería Estadística (EPIES)} \\
\multicolumn{2}{c}{Universidad Nacional de Ingeniería (UNI), Perú} \\
 %\verb+hrojas@uni.edu.pe+ &  \verb+lhumanchumo@uni.edu.pe+\\
\end{tabular}
\stepcounter{mpFootnoteValueSaver}%
    \footnotetext[\value{mpFootnoteValueSaver}]{%
      hrojas@uni.edu.pe}%
        \stepcounter{mpFootnoteValueSaver}%
    \footnotetext[\value{mpFootnoteValueSaver}]{%
      lhuamanchumo@uni.edu.pe}
\end{center}

\begin{abstract}
  \noindent The objective of this work is to apply the H\"anggi-Klimontovich stochastic differential equations to model and study the effects of anti-tumor chemotherapy in the case of continuous infusion delivering. The fluctuations generated by variations in drug concentration are modeled by the H\"anggi-Klimontovich stochastic integral. This integral,  which in the physics literature is sometimes called anti-It\^o integral,  in the last decade  it has been referenced quite as the more appropriate stochastic integral for model various biological and physical systems. Then, we make some comparisons with the model based on It\^o stochastic differential equations and their phase transitions that they generate, showing that the H\"anggi-Klimontovich  stochastic differential equations lead to more biologically realistic results.
 %  This note\footnote{Modified from ``WIC Symposium 2009 - Paper Instructions'' by
% F. Willems and T. Tjalkens} provides instructions for the preparation and
% submission of the final versions of the accepted papers for the
% 38$^{\textrm{th}}$ WIC\footnote{Werkgemeenschap voor Informatie- en Communicatietheorie}
% Symposium on Information Theory in the Benelux (SITB), and 7$^{\textrm{th}}$ joint
% WIC/IEEE SP Symposium on Information Theory and Signal Processing in the Benelux to be
% held in Enschede, the Netherlands, May 31--June 1, 2018   (see
% \verb+http://www.utwente.nl/sitb2018+).
\end{abstract}


\section{Introduction}\label{intro}
Cancerous tumors are caused by uncontrolled division of abnormal cells that continually multiply. In healthy cells, strict cell cycle regulation mechanisms prevent them from reproducing excessively, leading to normal cell growth \cite{genetic_human, text_oncology}. However, failures in cellular regulation, often caused by genetic predispositions and environmental factors, allow cells to escape cellular control and become carcinogenic cells through the accumulation of various genetic mutations, see Figure \ref{fig:cancer_growth}. In some cases, the cells that managed to escape cellular control are identified by cells of the immune system and are immediately eliminated. Unfortunately, in some other cases, depending on their accumulated mutations, these cancer cells manage to circumvent the immunological defense mechanism. If this occurs, the abnormal cells reproduce non-stop, a process known as cancerous proliferation. In the absence of drugs, proliferation continuous and damages surrounding tissues, spreads through the blood (metastasis), and usually causes death.
\begin{figure}[h]
     \centering
     \begin{subfigure}[b]{0.42\textwidth}
         \centering
         \includegraphics[width=\textwidth]{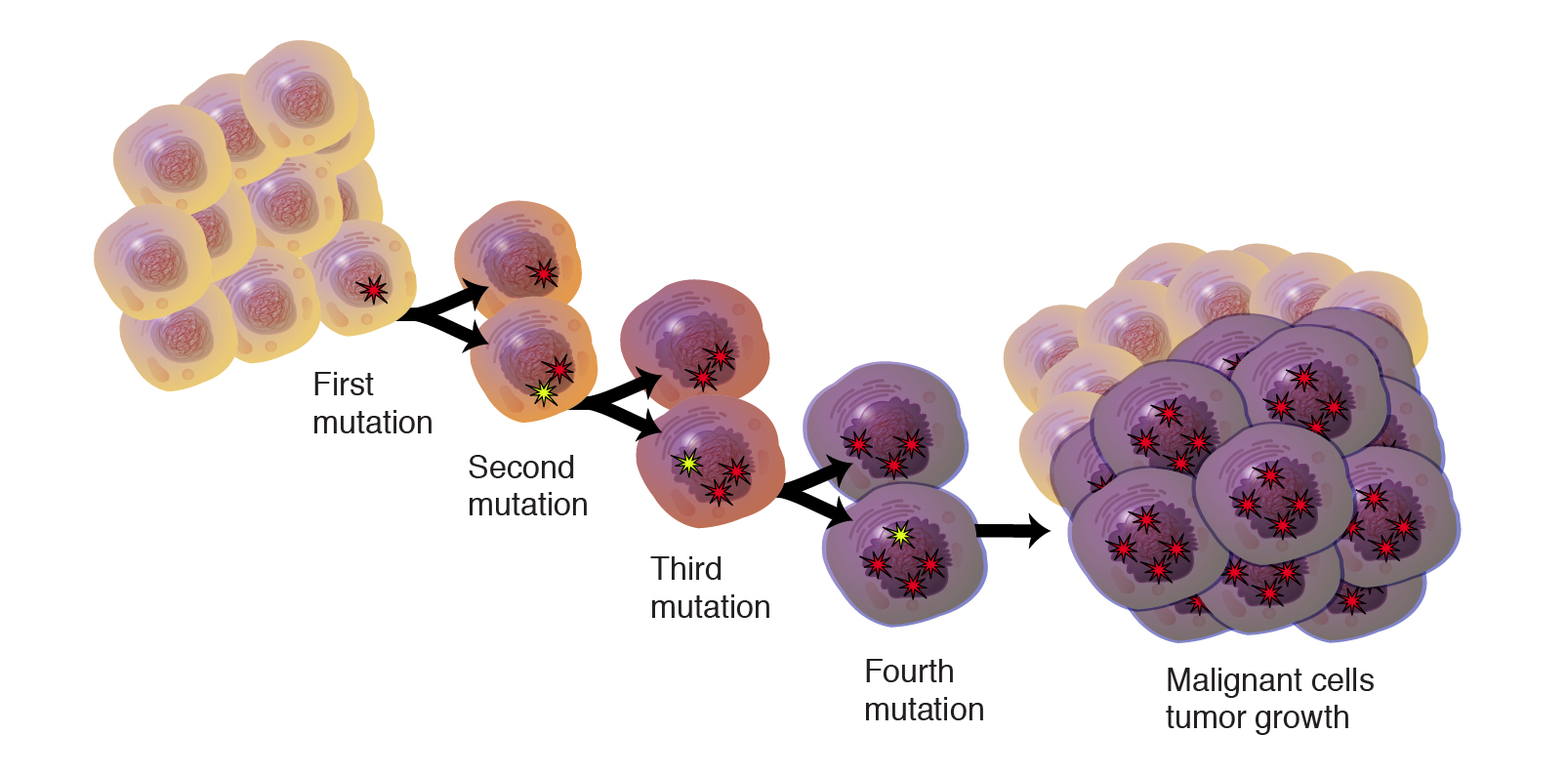}
          \caption{}
     \end{subfigure}
     \qquad
     \begin{subfigure}[b]{0.3\textwidth}
         \centering
         \includegraphics[width=\textwidth]{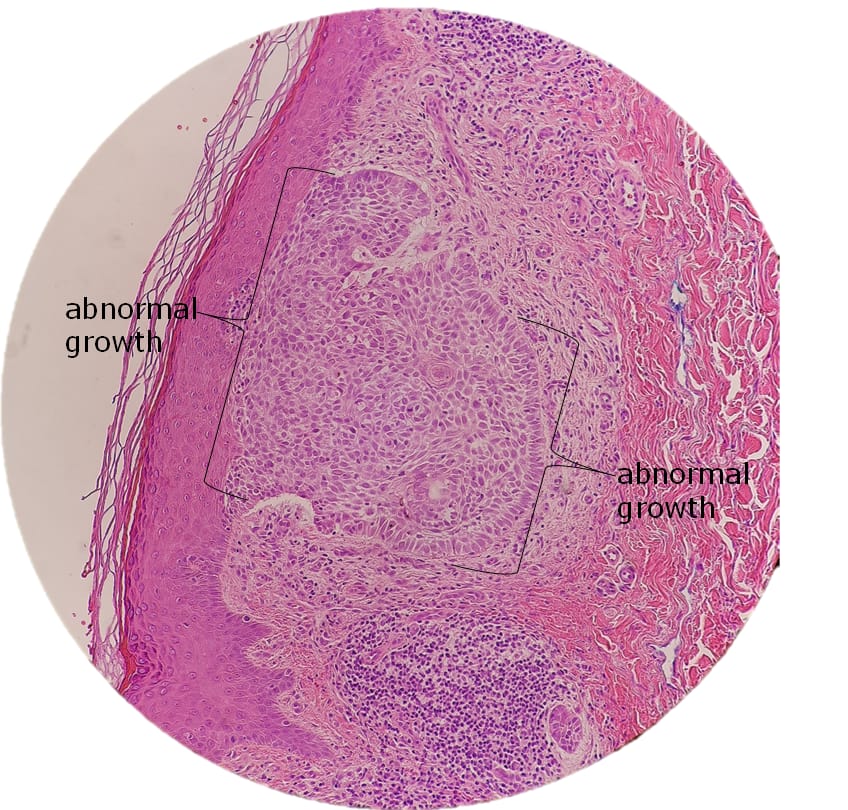}
         \caption{}
     \end{subfigure}
       \caption{(a) Uncontrolled mitosis at from genetic mutations. (b) Light microscope image of a cross section of tissue from a patient with cancer, abnormal cell growth is in brackets on either side. The shape of cancerous tissue is approximately circular.}
        \label{fig:cancer_growth}
\end{figure}

During the avascular phase, characterized by restrictions in the supply of nutrients and oxygen, abnormal growth is initially exponential. However, as the tumor grows, growth slows down and tends to stabilize, which depends on the logistic restrictions of its environment. This quasi-universal tumor dynamic, which is the same for different types of cancer, from different organs, with different phenotypes or genotypes, and both in vivo and in vitro, leads to a sigmoidal growth curve \cite{hart1998growth, norton2008cancer}. Based on these biomechanical characteristics of cancer,  many models have been proposed for tumor growth in the avascular phase, generally these models are expressed in terms of ordinary differential equations (ODE), of which we can highlight the logistic differential equation and the Gompertz differential equation \cite{wodarz2014dynamics}. 

One of the most flexible and robust of this family of models is the generalized logistic differential equation, for which the logistic and Gompertz models are particular cases. In this model, the number of cancer cells at time $t\geq 0$ is denoted by $x_t$ and, in the absence of drugs, its evolution is governed by the equation: $\mathrm{d}x_t=x_t(q-rx^v_t)\,\mathrm{d}t$, with $x_0>0$, where $q, r, v >0$ are parameters. Here $q$ is the exponential proliferation rate of first phase and $r$ is a parameter that characterizes the nutrient supply capacity of the region where the tumor is located. After diagnosis, to combat metastasis, patients undergo chemotherapy, in many cases, before and after a surgical intervention. In order to model the delivery and effect of the drug on the growth of cancer cells, the factor $g(\gamma(t))\cdot x_t$ is included, where $\gamma(t)$ is dose of the drug applied at time $t$ and $g(\gamma(t))$ is the destruction rate per cancer cell. Then, the cancer tumor growth model under influence of drugs is given by  $\mathrm{d}x_t=\big((q-rx^v_t)-g(\gamma(t))\big)\,x_t\,\mathrm{d}t$, with $x_0>0$, see \cite{krabs2007modelling}. 

An alternative for drug delivery, which partially reduces side effects, is continuous infusion of the drug. Therefore, that allows us to assume $g(\gamma(t))$ is approximately constant, so the anti-tumor chemotherapy model becomes
\begin{equation}\label{deterministic_model}
\frac{\mathrm{d}x_t}{\mathrm{d}t}=x_t(q-rx^v_t)-cx_t, \qquad x_t|_{t=0}=x_0>0,
\end{equation}
where $c>0$. As in many cases, if chemotherapy is prolonged for a long time, it is natural to be interested in the steady-state behavior of the deterministic model \eqref{deterministic_model}. It is easy to verify that this model has the following asymptotic properties:
\begin{itemize}
    \item If $q-c> 0$, then $x=0$ is an unstable fixed point and $x=\big(\frac{q-c}{r}\big)^{\frac{1}{v}}$ is an stable fixed point of the dynamical system generated by the ODE \eqref{deterministic_model}.
    \item However, if $q-c\leq 0$, then $x=0$ becomes a stable fixed point and attractor for the dynamical system.
\end{itemize}
\subsection{Random fluctuations in therapy delivery}
It is known that in most cases there are considerable variations in the concentration of the drug delivery to patients. Therefore, the constancy of $g(\gamma(t))$ in \eqref{deterministic_model} turns out to be just an idealization. To incorporate such random fluctuations due to variations in drug concentration, and sometimes also due to variations in the proliferation rate $q$, it is customary to modify model \eqref{deterministic_model} with the incorporation of multiplicative ``white noise", so the resulting model reads
\begin{equation}\label{SDE}
\frac{\mathrm{d}X_t}{\mathrm{d}t}=X_t(q-rX^v_t)-cX_t+\sigma X_t\,\xi_t, \qquad X_t|_{t=0}=X_0>0,
\end{equation}
$X_0$ a given random variable; of course, the solution $X_t$ has become a stochastic process rather than a deterministic function. Equation \eqref{SDE} is what is known in the literature as a stochastic differential equation (SDE). To cast Equation \eqref{SDE}  in the form of a well-defined SDE one should, first of all, consider its integral version
\begin{equation}\label{interpretation}
    	X_t=X_0+\int\limits_0^t \big(X_s(q-rX^v_s)-cX_s+\sigma X_s\big)\,\mathrm{d}s+ \sigma``\int\limits_0^t X_s\,\xi_s\,\mathrm{d}s\,\,",\quad t\geq 0.
\end{equation}
The second, and fundamental, step is to give a precise meaning to the integral between quotes, which cannot be considered neither as a Riemann nor as a Lebesgue integral. The two most famous stochastic integrals for this purpose are the It\^o integral and the Stratonovich integral, which lead to the It\^o calculus and Stratonovich calculus respectively. There has been a long debate, known as the It\^o versus Stratonovich dilemma, about which is the most appropriate integral to model dynamic systems disturbed with noise, a debate that is still open \cite{van1981ito, mannella2012ito}. Even this debate has been carried out in the field of biology \cite{turelli1977random}. However, the stochastic calculus that has been most used to model biological and oncological systems is the It\^o calculus. The notable preference towards the It\^o calculus has been due in large part to its good mathematical properties that are very useful when studying this type of systems. To understand and use the properties of It\^o calculus, see \cite{kuo, oksendal}.

If the It\^o calculus is adopted, Equation \eqref{interpretation} becomes the following It\^o-SDE
\begin{equation}\label{Ito_SDE_0}
\mathrm{d}X_t=\big(X_t(q-rX^v_t)-cX_t\big)\,\mathrm{d}t+\sigma X_t\,\mathrm{d}B_t, \qquad X_t|_{t=0}=X_0>0,
\end{equation}
where $B_t$ is a Brownian motion. Some versions of model \eqref{Ito_SDE_0}, adopting It\^o calculus, have been studied in \cite{lisei2008stochastic, donofrio}.

\subsection{It\^o calculus disadvantages in chemotherapy model}

In \cite{donofrio, donofrio2}, potential problems associated with the use of stochastic It\^o calculus in chemotherapy models are presented and discussed. In these works it is shown that It\^o calculus, the stochastic model \eqref{Ito_SDE_0}, leads to biologically paradoxical results . These can be summarized in two main elements:
\begin{itemize}
    \item If $q-c>0$, a noise with a sufficiently high intensity $\sigma$ can induce the tumor eradication almost surely. Specifically, if $\sigma^2\geq2(q-c)$, then tumor the eradication occurs with probability 1. Furthermore, noise-induced phase transitions in this regime appear to be synthetic and confusing biological meaning.
    
    \item If $q-c\leq 0$, the presence of noise does not change the result of the therapy with respect to the deterministic model \eqref{deterministic_model}.
\end{itemize}
The noise-induced tumor eradication, described in item one, is unrealistic and is easily questioned on biological grounds, for more details see \cite{donofrio, donofrio2}. 
\\

On the other hand, in the last decades a third stochastic calculus, given by the so called H\"anggi-Klimontovich integral, has been proposed as better adapted to describe certain biological and physical systems, particularly in statistical mechanics.  Recently, this integral was introduced in a mathematically precise way and its connections with the It\^o integral have been studied, see \cite{escudero2023versus}. This integral is not uniquely known as the H\"anggi-Klimontovich one; sometimes authors refer to it as the backward-It\^o, anti-It\^o, isothermal, or kinetic integral.

If the H\"anggi-Klimontovich calculus is adopted, Equation \eqref{interpretation} becomes the following H\"anggi-Klimontovich stochastic differential equation (HK-SDE)
\begin{equation}\label{HK_0}
\mathrm{d}X_t=\big(X_t(q-rX^v_t)-cX_t\big)\,\mathrm{d}t+\sigma X_t\bullet\mathrm{d}B_t, \qquad X_t|_{t=0}=X_0>0.
\end{equation}

In this paper, we are going to work with the H\"anggi-Klimontovich calculus and we are going to study the consequences and properties of that choice. We show that, in contrast to model \eqref{Ito_SDE_0}, model \eqref{HK_0} provides more biologically coherent results. We prove that for tumor-chemotherapy interaction models type \eqref{deterministic_model}, the H\"anggi-Klimontovich  calculus seems to be the most suitable for modeling the inclusion of random perturbations in contrast to the It\^o calculus.

\subsection{Mathematical background}
Under certain regularity conditions, there is an analytical equivalence between the It\^o and H\"anggi-Klimontovich formulations of an SDE.
\begin{theorem} [Conversion rule for HK-SDEs]\label{conversion}
Let $f:\mathbb{R}\times\mathbb{R}^{+} \longrightarrow \mathbb{R}$ and  $g:\mathbb{R}\times\mathbb{R}^{+} \longrightarrow \mathbb{R}$ be two functions continuous and satisfy the Lipschitz and linear growth conditions in $x$. Assume that $g$ is continuously differentiable such that $\frac{\partial g}{\partial x}g$ satisfies the Lipschitz and linear growth conditions in $x$. Then the unique solution to the It\^o-SDE
\begin{equation}\label{transformation_SDE}
\mathrm{d}X_t= \Big[f(X_t, t)+\frac{\partial g}{\partial x}(X_t,t)g(X_t, t)\Big]\,\mathrm{d}t+ g(X_t, t)\,\mathrm{d}W_t, \qquad X_0=x_0,
\end{equation}
solves the HK-SDE
\begin{equation}\label{HK-SDE}
\mathrm{d}X_t= f(X_t, t)\,\mathrm{d}t+g(X_t, t)\bullet\mathrm{d}W_t, \qquad X_0=x_0.
\end{equation}
almost surely. Correspondingly, there almost surely exists a unique solution to~\eqref{HK-SDE} that is given
by the solution to~\eqref{transformation_SDE}.
\end{theorem}
This theorem and its corresponding proof can be  found in \cite{escudero2023versus}. This conversion rule allows us to use It\^o's calculus rules and many of its other properties when we are working with HK-SDEs.

\section{HK-stochastic chemotherapy model} \label{}
Let $(\Omega,\mathcal{F},(\mathcal{F}_t)_{t\geq 0},\mathbb{P})$ be a completed filtered probability space in which the Brownian motion $(B_t)_{t\geq 0}$ is defined. To include random fluctuations in the continuous therapy infusion, we perturb Equation \eqref{deterministic_model} with a multiplicative noise term which is is modeled with H\"anggi-Klimontovich calculus. Therefore, the cancer cell process,  the number of cancer cells at time $t$, denoted as $(X_t)_{t \geq 0}$, evolves according to the HK-SDE 
\begin{equation}\label{Chemotherapy_SDE} 	
\mathrm{d}X_t=\big(qX_t-rX_t^{v+1}-cX_t\big)\mathrm{d}t+ \sigma X_t\bullet\mathrm{d}B_t, \qquad X_t|_{t=o}=X_0.
\end{equation}
for $0\leq t \leq T$, $\sigma >0$ and $X_0\in L^2(\Omega)$ is  $\mathcal{F}_0$ -- measurable. From Theorem \ref{conversion}, the Equation \eqref{Chemotherapy_SDE} is almost surely equivalent to Equation It\^o-SDE
\begin{equation}\label{ito_SDE}
\mathrm{d}X_t=\big(qX_t-rX_t^{v+1}-cX_t+\sigma^2 X_t\big)\mathrm{d}t+ \sigma X_t\,\mathrm{d}B_t, \qquad X_t|_{t=o}=X_0.
\end{equation}
Therefore, from Equation \eqref{ito_SDE} the transition probability density $p(x, t)$ of the diffusion process $X_t$ with space state $[0,\infty)$ satisfies the Fokker-Planck Equation (FPE) \cite{kuo}
 \begin{equation}\label{FPE1}
\partial_{t}\,p(x, t)=-\partial_{x}\big[\big(qx-rx^{v+1}-cx+\sigma^2 x\big)p(x,t)\big]+\frac{\sigma^2}{2}\partial_{xx}\,\big[x^2p(x,t)\big],
\end{equation}
with initial condition $$\lim\limits_{t\to 0+}p(x,t)=\delta(x_0-x).$$

\section{Asymptotic Analysis}
For interpretation purposes, the FPE \eqref{FPE1} can be rewritten as the following local conservation equation
\begin{equation}\label{continuity_equation}
	\partial_{t}p_t(x_0,x)+\partial_{x}J(x)=0 \quad\textrm{with} \quad \lim\limits_{t\to 0+}p_t(x_0,x)=\delta(x_0-x),
\end{equation}
where
\begin{equation*}
	J(x):= \big[\big(qx-rx^{v+1}-cx+\sigma^2 x\big)p(x,t)\big]-\frac{\sigma^2}{2}\partial_{x}\,\big[x^2p(x,t)\big].
\end{equation*}
Here $J(x)$ can be interpreted as a probability flow for which, in some cases, it is necessary to establish boundary conditions \cite{Pav, Gar}.

\begin{proposition}\label{Classfication}(Boundary Classification). If $q-c>0$, then the cancer cells process $(X_t)_{t \geq 0}$ is a diffusion process on $[0,\infty)$ for which $0$ is an exit boundary and $\infty$ is a natural boundary in the Feller's sense.
\end{proposition}
\begin{proof}
    Define the scale function of process $X_t$ as
\begin{equation}\label{scale_function}
    S(\xi)= \int\limits^\xi_1 \exp\Bigg\{-2\int\limits^\eta_1 \frac{q-rz^{v}-c+\sigma^2}{\sigma^2 z}\,\mathrm{d}z\Bigg\}\,\mathrm{d}\eta,
\end{equation}
and their speed density as
\begin{equation}\label{speed_density}
    m(\xi)= \frac{1}{(\sigma \xi)^2} \exp\Bigg\{2\int\limits^\xi_1 \frac{q-rz^{v}-c+\sigma^2}{\sigma^2 z}\,\mathrm{d}z\Bigg\}.
\end{equation}
Furthermore, let the speed measure $M$ induced by the speed density be, defined as follows
\begin{equation}\label{speed_measure}
   \textrm{d} M(\eta)=m(\eta)\,\textrm{d}\eta.
\end{equation}
To determine the nature of the boundary $\infty$, we need to calculate
\begin{equation*}   \Sigma(\infty)=\int\limits_1^\infty\Bigg[\int\limits_1^\xi\mathrm{d}M(\eta)\Bigg]\mathrm{d}S(\xi)\quad \textrm{and}\quad N(\infty)=\int\limits_1^\infty\Bigg[\int\limits_1^\eta\mathrm{d}S(\xi)\Bigg]\mathrm{d}M(\eta).
\end{equation*}
For both cases we combine equations \eqref{scale_function}, \eqref{speed_density} and \eqref{speed_measure}. For the first case, we have
\begin{eqnarray*}
 \Sigma(\infty)&=& \int\limits_1^\infty\Bigg[\int\limits_1^\xi \frac{\eta^{\frac{2}{\sigma^2}(q-c)}}{\sigma^2}\exp\bigg\{-\frac{2r}{v\sigma^2}(\eta^v-1)\bigg\}\,\textrm{d}\eta\Bigg]\exp\Bigg\{-2\int\limits^\xi_1 \frac{q-rz^{v}-c+\sigma^2}{\sigma^2 z}\,\mathrm{d}z\Bigg\}\textrm{d}\xi\\ \nonumber 
&&\qquad\geq\int\limits_1^\infty\Bigg[\int\limits_1^\xi \frac{\eta^{-\frac{2}{\sigma^2}(q-c)-1}}{\sigma^2}\,\textrm{d}\eta\Bigg]\textrm{d}\xi=\frac{1}{2(q-c)}\int\limits_1^\infty\Big[1-\xi^{-2\frac{(q-c)}{\sigma^2}}\Big]\textrm{d}\xi=+\infty,
\end{eqnarray*}
which allows us to deduce that $\Sigma(\infty)=\infty$. In a very similar way, for the second case we obtain that $N(\infty)=\infty$. Therefore, since $\Sigma(\infty)$ and $N(\infty)$ diverge simultaneously, then $\infty$ is a natural boundary in the Feller's sense, see Section 6 of Chapter 15 in \cite{Karlin}. In the sequel, we examine the  nature of the boundary point $0$. Accordingly,
\begin{equation*}
   \Sigma(0)=\int\limits_0^1\Bigg[\int\limits_\xi^1 m(\eta)\,\mathrm{d}\eta\Bigg]\mathrm{d}S(\xi)\leq \int\limits_0^1\Bigg[\int\limits_\xi^1 \frac{\eta^{\frac{2}{\sigma^2}(q-c)}}{\sigma^2}\,\textrm{d}\eta\Bigg]\textrm{d}\xi\leq\frac{1}{\sigma^2}\int \limits_0^1(1-\xi)\,\textrm{d}\xi<+\infty, 
\end{equation*}

\begin{eqnarray*}
N(0)&=&\int\limits_0^1\Bigg[\int\limits_\eta^1\mathrm{d}S(\xi)\Bigg]m(\eta)\,\mathrm{d}\eta\geq\int\limits_0^1\Bigg[\int\limits_\eta^1\exp\bigg\{-2\,\bigg(\frac{q-c+\sigma^2}{\sigma^2}\bigg)\ln\xi\bigg\}\,\textrm{d}\xi\Bigg]m(\eta)\,\mathrm{d}\eta\\ \nonumber 
&&\quad\geq \int\limits_0^1\Bigg[\int\limits_\eta^1\frac{1}{\xi^{2\,(\frac{q-c}{\sigma^2})+2}}\,\textrm{d}\xi\Bigg]\,\mathrm{d}\eta\geq \int\limits_0^1\Bigg[\int\limits_\eta^1\frac{1}{\xi^{2}}\,\textrm{d}\xi\Bigg]\,\mathrm{d}\eta= \int\limits_0^1(\eta^{-1}-1)\,\mathrm{d}\eta=+\infty.
\end{eqnarray*}
Then for the boundary 0 we get $\Sigma(0)<\infty$ and $N(0)=\infty$. Therefore, we can conclude that $0$ is an exit boundary in the Feller's sense \cite{Karlin}.
\end{proof}

Based on the classification established in Proposition \ref{Classfication}, we know that the FPE \eqref{continuity_equation}, admits a unique solution generated for its respective fundamental solution, for which it is not necessary to impose boundary conditions \cite{Feller}. In the following, we explicitly determine the stationary probability density of the FPE that characterizes the steady-state behavior of the process $X_t$, subsequently, we discuss its main properties.

\begin{proposition}\label{prop_stationary}
  Under the assumption established in Proposition \ref{Classfication}, i.e., $q-c>0$, the cancer cells process $(X_t)_{t\geq 0}$ admits a stationary distribution $p_s(x)$ on the unrestricted domain $[0, \infty)$. Moreover, $p_s(x)$ is the unique time-independent probability distribution that solves the Fokker-Planck equation \eqref{continuity_equation} on $[0, \infty)$ and it is given by the explicit formula
	\begin{equation}\label{Stationary_Function}
		p_s(x)=\Bigg[\frac{v}{\Gamma\Big(\frac{2(q-c)+\sigma^2}{v\,\sigma^2}\Big)}\,\bigg(\frac{2r}{v\,\sigma^2}\bigg)^{\frac{1}{v}\big(\frac{2(q-c)}{\sigma^2}+1\big)}\Bigg]\,e^{-\frac{2rx^v}{v\sigma^2}}x^{2\,(\frac{q-c}{\sigma^2})},
	\end{equation}
where $\Gamma(x)$ is the gamma function.
\end{proposition}
\begin{proof}
    First, we compute the explicit solution formula. Every  time-independent probability distribution satisfies $\partial_tp_s(x)=0$, Therefore, the Equation \eqref{continuity_equation} can be written in terms of the probability flow as $\partial_x J(x)=0$. As a consequence of Proposition \ref{Classfication}, this equation does not need boundaries conditions, so the equation reduces to the second-order ordinary differential equation
    \begin{equation*}
	\frac{\textrm{d}}{\textrm{d}x}\big[\big(qx-rx^{v+1}-cx+\sigma^2 x\big)p_s(x)\big]=\frac{\sigma^2}{2}\frac{\textrm{d}^2}{\textrm{d} x^2}\,\big[x^2p_s(x\big)], 
\end{equation*}
from which the valid fundamental solution is obtained
\begin{equation}\label{fundamental_solution}
		p_s(x)=k_0\, e^{-\frac{2rx^v}{v\sigma^2}}x^{2\,(\frac{q-c}{\sigma^2})},
	\end{equation}
 where $k_0$ is the normalization constant. In the sequel, to guarantee the existence and uniqueness of the solution \eqref{fundamental_solution}, we show that $k_0$ is finite
 \begin{equation}\label{k_0}
     k_0^{-1}=\int\limits^\infty_0 e^{-\frac{2rx^v}{v\sigma^2}}x^{2\,(\frac{q-c}{\sigma^2})}\,\textrm{d}x= -\frac{1}{v}\,\bigg(\frac{2r}{v\,\sigma^2}\bigg)^{-\frac{1}{v}\big(\frac{2(q-c)}{\sigma^2}+1\big)} {\Gamma\Bigg(\frac{\frac{2(q-c)}{\sigma^2}+1}{v}, \frac{2rx^v}{v\sigma^2}\Bigg)}\Bigg|_0^\infty,
 \end{equation}
 where $\Gamma(s,x)$ is the incomplete gamma function, for which we have the very useful asymptotic property $\frac{\Gamma(s,x)}{x^{s-1}\,e^{-x}}\longrightarrow 1$ as $x\longrightarrow\infty$, see \cite{Handbook}. Making use of this property in Equation \eqref{k_0}, we obtain
\begin{eqnarray*}
 k_0^{-1}&=& -\frac{1}{v}\,\bigg(\frac{2r}{v\,\sigma^2}\bigg)^{-\frac{1}{v}\big(\frac{2(q-c)}{\sigma^2}+1\big)} \Bigg[\lim\limits_{x\to \infty}{\Gamma\Bigg(\frac{\frac{2(q-c)}{\sigma^2}+1}{v}, \frac{2rx^v}{v\sigma^2}\Bigg)}-{\Gamma\Bigg(\frac{\frac{2(q-c)}{\sigma^2}+1}{v}, 0\Bigg)\Bigg]}\\ \nonumber
 &=& -\frac{1}{v}\,\bigg(\frac{2r}{v\,\sigma^2}\bigg)^{-\frac{1}{v}\big(\frac{2(q-c)}{\sigma^2}+1\big)} \Bigg[\lim\limits_{x\to \infty}\bigg({\frac{2rx^v}{v\sigma^2}}\bigg)^{\Big(\frac{\frac{2(q-c)}{\sigma^2}+1}{v}-1\Big)}e^{-{\frac{2rx^v}{v\sigma^2}}}-{\Gamma\Bigg(\frac{\frac{2(q-c)}{\sigma^2}+1}{v}\Bigg)\Bigg]}\\ \nonumber
 &=&\frac{1}{v}\,\bigg(\frac{2r}{v\,\sigma^2}\bigg)^{-\frac{1}{v}\big(\frac{2(q-c)}{\sigma^2}+1\big)} \Gamma\Bigg(\frac{\frac{2(q-c)}{\sigma^2}+1}{v}\Bigg).
\end{eqnarray*}
Therefore, under the assumptions this proposition, we have 
\begin{equation}\label{k_0_expplicito}
    k_0=\frac{v}{\Gamma\Big(\frac{\frac{2(q-c)}{\sigma^2}+1}{v}\Big)}\,\bigg(\frac{2r}{v\,\sigma^2}\bigg)^{\frac{1}{v}\big(\frac{2(q-c)}{\sigma^2}+1\big)}<\infty.
\end{equation}
 To finish the proof, Equation \eqref{k_0_expplicito} is substituted into Equation \eqref{fundamental_solution}.
\end{proof} 
\begin{figure}[h]
\centering
\includegraphics[width=0.7\textwidth]{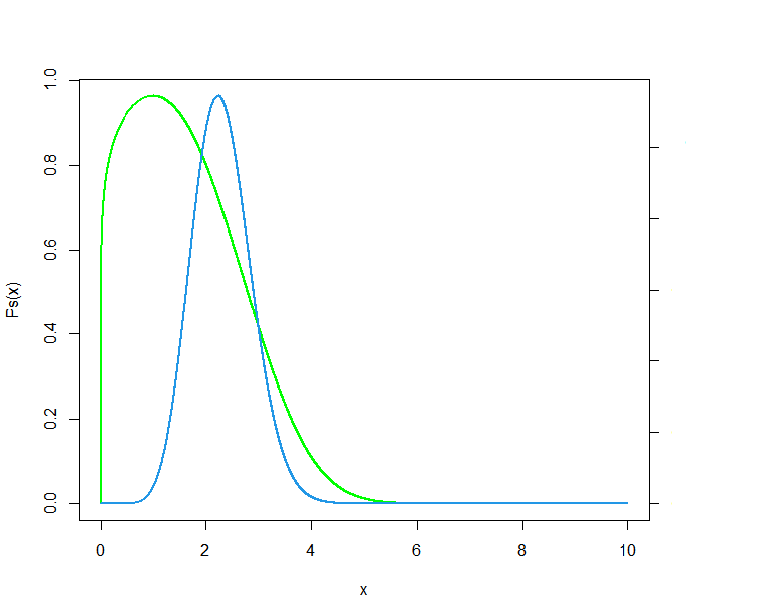}
\caption{Stationary distribution of  $X_t$ for $q-c>0$ and different values of noise intensity $\sigma^2$. (Light blue curve) $\sigma^2>2(q-c)$. (Green line) $\sigma^2\leq2(q-c)$. }
\label{fig_non_transition}
\end{figure}
\begin{remark}
    In the case that the stationary distribution \eqref{Stationary_Function} does not exist, zero is stationary point: drift and diffusion vanish simultaneously for $x = 0$. Then zero becomes is a attractor point, the stationary probability mass will be entirely concentrated on zero. Speaking in terms of probability one can say that $p_s(x)=\delta(x)$.
\end{remark}
\section{Non-equilibrium phase transitions} \label{subm}
It is known that the critical points of the stationary distribution $p_s(x)$ are the most appropriate indicators to identify phase transitions in the steady-state behavior of stochastic process $X_t$ \cite{Lefever1984}. In the following theorem we determine the critical points of $p_s(x)$. Furthermore, we establish an interesting connection between the steady-state behavior of the deterministic dynamical system governed by equation \eqref{deterministic_model} and the stationary distribution of its stochastic counterpart governed by the HK-SDE \eqref{Chemotherapy_SDE}.

\begin{theorem}\label{equilibrium}(Robustness of deterministic behavior). if $q-c>0$, then the critical points of the stationary
distribution \eqref{Stationary_Function}, i.e., the global minimum and maximum points of $p_s(x)$,  coincides with the fixed points of the dynamical
system generated by \eqref{deterministic_model}
\begin{equation}\label{dynamic_system}
\frac{\mathrm{d}x_t}{\mathrm{d}t}=x_t(q-rx^v_t)-cx_t.
\end{equation}
Furthermore, the global maxima point of \eqref{Stationary_Function} coincide with the stable fixed points of \eqref{dynamic_system}. Correspondingly, the global minima point of \eqref{Stationary_Function} coincide with the unstable fixed points of \eqref{dynamic_system}.
\end{theorem}
\begin{proof}
    First of all note that under the current assumptions, equation \eqref{dynamic_system} presents unique and global-in-time solutions to its associated initial value problems; therefore it generates a well-defined dynamical system. Moreover, these assumptions guarantee that $p_s(x)$ is differentiable; thus by direct differentiation of \eqref{Stationary_Function} we get
    \begin{equation}\label{dif_stationary}
        \frac{\mathrm{d}p_s(x)}{\mathrm{d}x}= \frac{2}{\sigma^2}\Bigg[\frac{v}{\Gamma\Big(\frac{2(q-c)+\sigma^2}{v\,\sigma^2}\Big)}\,\bigg(\frac{2r}{v\,\sigma^2}\bigg)^{\frac{1}{v}\big(\frac{2(q-c)}{\sigma^2}+1\big)}\Bigg]\,x^{\frac{2(q-c)}{\sigma^2}-1}\,e^{-\frac{2rx^v}{v\sigma^2}}(q-c-rx^v).
    \end{equation}
    From Equation \eqref{dif_stationary}, notice that $x=0$ and $x=(\frac{q-c}{r})^{\frac{1}{v}}$ are critical points of $p_s(x)$, which correspond to the fixed points of \eqref{dynamic_system}, see Section \ref{intro}. On the one hand, it is known that $x=0$ is a unstable fixed points and $x=(\frac{q-c}{c})^{\frac{1}{v}}$ is a stable fixed points for the dynamical
    system generated by \eqref{dynamic_system}. On the other hand, it is easy to verify that 
    \begin{equation*}
       \frac{\mathrm{d}^2p_s(x)}{\mathrm{d}x^2}\bigg|_{x=(\frac{q-c}{r})^{\frac{1}{v}}}<0 \qquad\textrm{and}\qquad p_s(0)=0.
    \end{equation*}
    Since $p_s(x)\geq 0$ for all $x \in [0,\infty)$, we conclude that $x=0$ is the global minimum and $x=(\frac{q-c}{r})^{\frac{1}{v}}$ is the global maximum of $p_s(x)$.
\end{proof}

\begin{remark}
  This sort of robustness of the deterministic behavior is the consequence of the H\"anggi-Klimontovich calculus. Other stochastic calculus, such as It\^o or Stratonovich, do not possess this property, see Chapter 6 in \cite{Lefever1984}. 
 Note that this does not mean in either case that the solution to the stochastic equation is simply the noisy counterpart of the solution to the deterministic equation. For instance, transitions between different stable states, which are forbidden in the absence of noise, may happen in the presence of stochastic forcing. 
\end{remark}

From Theorem \ref{equilibrium} we conclude that: If $q-c>0$, the point $x=(\frac{q-c}{r})^{\frac{1}{v}}$ is always the maximum of $p_s(x)$, which corresponds to the most probable value. This is so, regardless of the value assumed by the intensity of fluctuations environment $\sigma$. In the same way, independent of $\sigma$, the point $x=0$ is just an unstable fixed point, which corresponds to the least probable value. In other words, the presence of noise does not alter the therapy result obtained through the deterministic model. Therefore, if $q-c>0$ then there are no abrupt changes in the shape of the stationary distribution, that is, there are no noise- induced phase transitions, see Figure \ref{fig_non_transition}. 
\begin{figure}[h]
\centering
\includegraphics[width=0.7\textwidth]{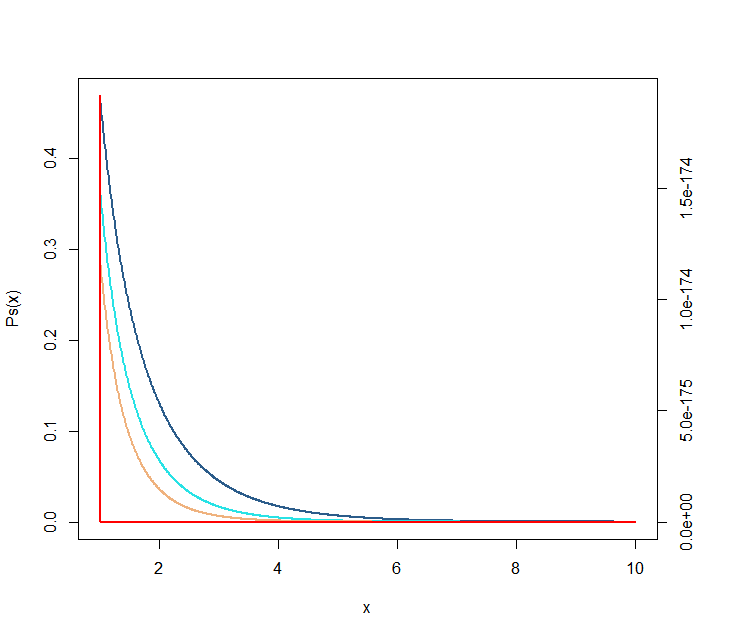}
\caption{Stationary distribution of $X_t$ for $q-c\leq 0$ and different values of noise intensity $\sigma^2$. (Blue, light blue and orange curve) $\sigma^2>2(c-q)$. (Red line) $\sigma^2\leq2(c-q)$.} 
\label{fig_transition}
\end{figure}

\begin{corollary}\label{phase_transition}(Noise-induced eradication). If $q-c\leq 0$, the stationary distribution $p_s(x)$ has a transition at $\sigma^2=2(c-q)$. Therefore, if $\sigma^2>2(c-q)$ then $p_s(x)$ is decreasing, for which the point $x=0$ is its global maximum. Correspondingly, if $\sigma^2\leq2(c-q)$ then $p_s(x)$ is unbounded and entirely concentrated on zero, i.e., $p_s(x)=\delta(x)$.
\end{corollary}

The noise-induced phase transition described in Corollary \ref{phase_transition} corresponds to an abrupt change in the shape of the stationary distribution $p_s(x)$: If $\sigma^2>2(c-q)$, then the stationary distribution is a genuine distribution, for which $x=0$ is the most probable value. On the other hand, if $\sigma^2 \leq 2(c-q)$, then the stationary distribution is a degenerate distribution, for which $x=0$ is an attractor point, see Figure \ref{fig_transition}. In other words, at $\sigma^2=2(c-q)$, as the intensity of the fluctuations decreases, there is a transition between the ``small tumor highly probable regime'' to the `` tumor eradication almost surely regime'' . This shows that, under suitable conditions, the cancer cell process $X_t$ can be made to undergo a transition by increasing or decreasing the intensity of the noise fluctuations. 
\section{Conclusions}

Theorem \ref{equilibrium} and Corollary \ref{phase_transition} have coherent biological consequences: If $q-c\leq 0$, then the most probable steady state for tumor size corresponds to non-zero values for $\sigma^2<2(p-q)$, while that for $\sigma^2\geq 2(p-q)$, the tumor eradication occurs with probability 1. However, the case $q-c>0$ has a more interesting consequence since noise-induced tumor eradication cannot occur with probability 1, this is true for any value assumed for the noise intensity $\sigma^2$. This scenario is profoundly different from that  “noise-induced tumor eradication” presented in \cite{donofrio, donofrio2} and is due to the use of the H\"anggi-Klimontovich calculus. Consequently, our results seem to be more realistic from a biological aspect.

In work \cite{donofrio}, the authors recommend the use of Fuzzy theory and the theory of differential equations with bounded noise perturbations to eliminate the presence of the “noise-induced tumor eradication”, considering it synthetic and without coherent biological meaning. But they also admit that both suggested theories still do not have the necessary theoretical development or lose the desirable properties of It\^o calculus. It should be noted that our proposal, in addition to making greater biological sense, partly preserves desirable properties of the It\^o calculus since there is a very close and formally defined relationship between the It\^o calculus and the H\"anggi-Klimontovich  calculus \cite{escudero2023versus}.

\section*{Acknowledgments}
This work has been partially supported by the Research Institute of the Faculty of Economics, Statistics and Social Sciences (IECOS) and the Vice-Rectorate of Research of the National University of Engineering (VRI-UNI) through the research project 25-2023-002472.

\end{document}